\documentclass{amsart}
\usepackage[utf8]{inputenc}

\usepackage{amssymb}
\usepackage{amsmath}
\usepackage{amsthm}
\usepackage{mathrsfs}
\usepackage{accents}
\usepackage[pdftex]{graphicx}
\usepackage{tikz}
\usepackage{url}
\usepackage{mathtools}
\usepackage{mdframed}
\usepackage{geometry}
\usepackage[backend=biber,style=alphabetic,sorting=nty,doi=false,isbn=false,url=false,eprint=true]{biblatex}
\renewbibmacro{in:}{}
\usepackage{hyperref}

\newcommand{\scrN}{\mathcal{N}}
\newcommand{\scrI}{\mathcal{I}}

\newcommand{\scrJ}{\mathcal{J}}

\renewcommand{\P}{\mathbb{P}}

\newcommand{\R}{\mathbb{R}}

\newcommand{\range}{\operatorname{ran}}

\newcommand{\append}{{}^\frown}

\newcommand\forces{\Vdash}

\newcommand{\HDZ}{\mathrm{HDZ}}
\newcommand{\dimh}{\dim_{\mathrm{H}}}
\newcommand{\Haus}{\mathcal{H}}
\newcommand{\non}{\operatorname{non}}
\newcommand{\cov}{\operatorname{cov}}
\newcommand{\add}{\operatorname{add}}
\newcommand{\cof}{\operatorname{cof}}
\newcommand{\Cof}{\mathbf{Cof}}
\newcommand{\Cov}{\mathbf{Cov}}

\newcommand{\Lc}{\mathbf{Lc}}
\newcommand{\nul}{\mathrm{null}}
\newcommand{\id}{\mathrm{id}}
\newcommand{\diam}{\mathrm{diam}}
\newcommand{\height}{\operatorname{ht}}
\newcommand{\pow}{\mathrm{pow}}

\newcommand{\twototheltomega}{2^{<\omega}}

\newcommand{\from}{\colon}
\newcommand{\seq}[1]{{\langle#1\rangle}}
\newcommand{\AND}{\mathop{\&}}

\newcommand{\wLc}{\mathbf{wLc}}
\newcommand{\pr}{\operatorname{pr}}
\newcommand{\PR}{\operatorname{PR}}
\newcommand{\frakc}{\mathfrak{c}}

\newcommand{\frakv}{\mathfrak{v}}

\DeclarePairedDelimiterX{\norm}[1]{\lVert}{\rVert}{#1}
\DeclarePairedDelimiter\abs{\lvert}{\rvert}
\DeclarePairedDelimiter\floor{\lfloor}{\rfloor}
\DeclarePairedDelimiter\ceil{\lceil}{\rceil}

\renewcommand\emptyset{\varnothing}
\renewcommand\subset{\subseteq}
\renewcommand{\setminus}{\smallsetminus}

\theoremstyle{definition}
\newtheorem{thm}{Theorem}[section]
\newtheorem*{thm*}{Theorem}
\newtheorem{defi}[thm]{Definition}
\newtheorem*{defi*}{Definition}
\newtheorem{lem}[thm]{Lemma}
\newtheorem*{lem*}{Lemma}
\newtheorem{fact}[thm]{Fact}
\newtheorem*{fact*}{Fact}
\newtheorem*{formula*}{Formula}
\newtheorem{prop}[thm]{Proposition}
\newtheorem*{prop*}{Proposition}

\newtheorem*{exm*}{Example}
\newtheorem{rmk}[thm]{Remark}
\newtheorem*{rmk*}{Remark}
\newtheorem{cor}[thm]{Corollary}
\newtheorem*{cor*}{Corollary}
\newtheorem*{notation*}{Notation}
\newtheorem{problem}[thm]{Problem}
\newtheorem{conj}[thm]{Conjecture}

\title{Cardinal invariants associated with Hausdorff measures}
\author{Tatsuya Goto}
\date{Decemeber 15, 2021}
\address{
	\newline
	Graduate School of Information Science\newline
	Nagoya University\newline
	Furo-cho, Chikusa-ku, Nagoya 464-8601\newline
	JAPAN
}

\email{goto.tatsuya@k.mbox.nagoya-u.ac.jp}

\begin{document}
	\maketitle
	
	\begin{abstract}
		We consider cardinal invariants determined from Hausdorff measures.
		We separate many cardinal invariants of Hausdorff measure $0$ ideals using two models that separate many cardinal invariants of Yorioka ideals at once from earlier work.
		Also we show the uniformity numbers of $s$-dimensional Hausdorff measure $0$ ideals for $0 < s < 1$ and that of Lebesgue null ideal can be separated using the Mathias forcing.
	\end{abstract}
	
	\section{Basic Definitions}
	
	\begin{defi}
		$\nul$ denotes the Lebesgue measure 0 ideal for $2^\omega$.
	\end{defi}
	
	\begin{defi}
		A function $f\from [0, \infty) \to [0, \infty)$ is a {\bfseries gauge function} if $f(0) = 0, \lim_{x \to 0} f(x) = 0$ and $f$ is nondecreasing.
		
		Let $X$ be a metric space.
		For $A \subset X$, $f$ a gauge function and $\delta \in (0, \infty]$, we define
		\[
		\Haus^f_\delta(A) = \inf \{ \sum_{n=0}^\infty f(\diam(C_n)) : C_n \subset X \text{ (for $n \in \omega$)} \text{ with } A \subset \bigcup_{n \in \omega} C_n \text{ and } (\forall n) (\diam(C_n) \le \delta) \}.
		\]
		Next, for $A \subset X$ and $f$ a gauge function, we define
		\[
		\Haus^f(A) = \lim_{\delta \to 0} \Haus^f_\delta(A).
		\]
		We call $\Haus^f(A)$ the {\bfseries Hausdorff measure} with gauge function $f$. In particular, for $A \subset X$ and $s > 0$, let
		\[
		\Haus^s(A) = \Haus^{\pow_s}(A)
		\]
		where $\pow_s(x) = x^s$. For $A \subset X$, let
		\[
		\dimh(A) = \sup \{ s : \Haus^s(A) = \infty \} = \inf \{ s : \Haus^s(A) = 0 \}.
		\]
		We call $\dimh(A)$ the {\bfseries Hausdorff dimension} of $A$.
		
		We metrize the Cantor space $2^\omega$ by
		\[
		d(x, y) = \begin{cases} 0 & \text{(if $x = y$),} \\
			2^{-\min\{n : x(n) \ne y(n) \}} & \text{(otherwise)}. \end{cases}
		\]
	\end{defi}
	
	\begin{defi}
		\begin{enumerate}
			\item For a metric space $X$, define $\HDZ_X = \{ A \subset X : \dimh(A) = 0 \}$.
			\item Define $\HDZ = \HDZ_{2^\omega}$.
		\end{enumerate}
	\end{defi}
	
	\begin{defi}
		For a metric space $X$ and a gauge function $f$, define $\scrN^f_X = \{ A \subset X : \Haus^f(A) = 0 \}$. Especially we define $\scrN^f = \scrN^f_{2^\omega}$.
		For $s > 0$, define $\scrN^s_X = \scrN^{\pow_s}_X$ and $\scrN^s = \scrN^{\pow_s}_{2^\omega}$.
	\end{defi}
	
	\begin{rmk}
		\begin{enumerate}
			\item $\scrN^1 = \nul$.
			\item $\HDZ = \bigcap_{s > 0} \scrN^s$.
		\end{enumerate}
	\end{rmk}
	
	\begin{defi}
		For $\sigma \in (\twototheltomega)^\omega$, define $\height \sigma \from \omega \to \omega$ and $[\sigma]_\infty \subset 2^\omega$ by
		\[
		(\height \sigma)(n) = \abs{\sigma(n)} \text{ and }
		\]
		\[
		[\sigma]_\infty = \{ x \in 2^\omega : (\exists^\infty n) \sigma(n) \subset x \}.
		\]
		For $g \in \omega^\omega$, define
		\[
		\scrJ_g = \{ A \subset 2^\omega : (\exists \sigma \in (\twototheltomega)^\omega) (\height\sigma = g \AND \subset [\sigma]_\infty) \}.
		\]
		For $f, g\in \omega^\omega$, define
		\[
		f \ll g \iff (\forall k \in \omega) (f \circ \pow_k \le^* g)
		\]
		For $f \in \omega^\omega$ increasing, define
		\[
		\scrI_f = \bigcup_{g \gg f} \scrJ_g.
		\]
		We call $\scrI_f$ the Yorioka ideal for $f$.
	\end{defi}
	
	\begin{defi}
		\begin{enumerate}
			\item If $X, Y$ are sets and $R$ is a subset of $X \times Y$, we call a triple $(X, Y, R)$ a relational system.
			\item For a relational system $\mathcal{A} = (X, Y, R)$, define $\mathcal{A}^\perp = (Y, X, \hat{R})$, where $\hat{R} = \{ (y, x) \in Y \times X : \neg (x \mathrel{R} y)\})$.
			\item For a relational system $\mathcal{A} = (X, Y, R)$, define $\norm{\mathcal{A}} = \min \{ \abs{B} : B \subset Y \land (\forall x \in X)(\exists y \in B) (x \mathrel{R} y) \}$.
			\item For relational systems $\mathcal{A} = (X, Y, R), \mathcal{B} = (X', Y', S)$, we call a pair $(\varphi, \psi)$ a Galois--Tukey morphism from $\mathcal{A}$ to $\mathcal{B}$ if $\varphi \from X \to X'$, $\psi \from Y' \to Y$ and $(\forall x \in X)(\forall y \in Y')(\varphi(x) \mathrel{S} y \implies x \mathrel{R} \psi(y))$.
		\end{enumerate}
	\end{defi}
	
	\begin{fact}[{{\cite[Theorem 4.9]{blass2010combinatorial}}}]
		If there is a Galois--Tukey morphism $(\varphi, \psi)$ from $\mathcal{A}$ to $\mathcal{B}$, then $\norm{\mathcal{A}} \le \norm{\mathcal{B}}$ and $\norm{\mathcal{B}^\perp} \le \norm{\mathcal{A}^\perp}$. 
	\end{fact}
	
	\begin{defi}
		\begin{enumerate}
			\item For $c \in (\omega+1)^\omega, h \in \omega^\omega$, define $\prod c = \prod_{n \in \omega} c(n)$ and $S(c, h) = \prod_{n \in \omega} [c(n)]^{\le h(n)}$.
			\item For $x \in \prod c$ and $\varphi \in S(c, h)$, define $x \in^* \varphi$ iff $(\forall^\infty n)(x(n) \in \varphi(n))$ and define $x \in^\infty \varphi$ iff $(\exists^\infty n)(x(n) \in \varphi(n))$.
		\end{enumerate}
	\end{defi}
	
	\begin{defi}
		\begin{enumerate}
			\item For $c \in (\omega+1)^\omega, h \in \omega^\omega$, define $\Lc(c, h) = (\prod c, S(c, h), \in^*)$, $\frakc^\forall_{c, h} = \norm{\Lc(c, h)}$ and $\frakv^\forall_{c, h} = \norm{\Lc(c, h)^\perp}$.
			\item Define $\wLc(c, h) = (\prod c, S(c, h), \in^\infty)$, $\frakc^\exists_{c, h} = \norm{\wLc(c, h)}$ and $\frakv^\exists_{c, h} = \norm{\wLc(c, h)^\perp}$.
			\item For an ideal $I$ on $X$, define $\Cov(I) = (X, I, \in)$, $\cov(I) = \norm{\Cov(I)}$ and $\non(I) = \norm{\Cov(I)^\perp}$.
			\item For an ideal $I$ on $X$, define $\Cof(I) = (I, I, \subset)$, $\cof(I) = \norm{\Cof(I)}$ and $\add(I) = \norm{\Cof(I)^\perp}$.
		\end{enumerate}
	\end{defi}
	
	\section{Stability under changing underlying spaces}
	
	\begin{defi}
		Let $X$ and $Y$ be metric spaces and $\alpha, c > 0$.
		A map $f\from X \to Y$ is said to be {\bfseries $\alpha$-H\"{o}lder} with constant $c$ if for all $x, x' \in X$ we have $d(f(x), f(x')) \le c \cdot d(x, x')^\alpha$.
		A map $f\from X \to Y$ is said to be {\bfseries $\alpha$-co-H\"{o}lder} with constant $c$ if for all $x, x' \in X$ we have $d(f(x), f(x')) \ge c \cdot d(x, x')^\alpha$.
		A map $f\from X \to Y$ is said to be {\bfseries $\alpha$-bi-H\"{o}lder} with constants $c_1, c_2$ if it is both $\alpha$-H\"{o}lder with constant $c_1$ and $\alpha$-co-H\"{o}lder with constant $c_2$.
	\end{defi}
	
	\begin{prop}\label{holderhaus}
		Let $X$ and $Y$ be metric spaces and $\alpha > 0$.
		\begin{enumerate}
			\item If there is $\alpha$-H\"{o}lder map $f\from X \to Y$ with constant $c$, then for all $s > 0$ we have $\Haus^{s/\alpha}(f(X)) \le c^{s/\alpha} \Haus^s(X)$ and $\dimh f(X) \le (1/\alpha) \dimh X$.
			\item If there is $\alpha$-co-H\"{o}lder map $f\from X \to Y$ with constant $c$, then for all $s > 0$ we have $\Haus^{s/\alpha}(f(X)) \ge c^{s/\alpha} \Haus^s(X)$ and $\dimh f(X) \ge (1/\alpha) \dimh X$.
			\item If there is $\alpha$-bi-H\"{o}lder map $f\from X \to Y$ with constant $c_1, c_2$, then for all $s > 0$ we have $c_1^{s/\alpha} \Haus^s(X) \le \Haus^{s/\alpha}(f(X)) \le c_2^{s/\alpha} \Haus^s(X)$ and $\dimh f(X) = (1/\alpha) \dimh X$.
		\end{enumerate}
	\end{prop}
	\begin{proof}
		Item 1.
		Let $\delta > 0$ and $\seq{C_n : n \in \omega}$ be a $\delta$-cover of $X$.
		Then $\seq{f(C_n) : n \in \omega }$ is a cover of $f(X)$ and the diameter of each member satisfies
		\[
		\diam(f(C_n)) \le c \cdot \diam(C_n)^\alpha \le c \cdot \delta^\alpha =: \varepsilon.
		\]
		So $\seq{f(C_n) : n \in \omega }$ is a $\varepsilon$-cover of $f(X)$.
		Thus
		\[
		\Haus^{s/\alpha}_\varepsilon(f(X)) \le \sum_{n} \diam(f(C_n))^{s/\alpha} \le \sum_n c^{s/\alpha} \cdot \diam(C_n)^s.
		\]
		Take the infimum for $(C_n)$ we get the following.
		\[
		\Haus^{s/\alpha}_\varepsilon(f(X)) \le c^{s/\alpha} \Haus^s_\delta(X).
		\]
		Letting $\delta$ tend to 0, we have
		\[
		\Haus^{s/\alpha}(f(X)) \le c^{s/\alpha} \Haus^s(X).
		\]
		
		In order to prove the dimension inequality, Let $s > \dimh X$.
		Then $\Haus^s(X) = 0$, so $\Haus^{s/\alpha}(f(X))$ is also equal to $0$. Thus $s/\alpha \ge  \dimh f(X)$.
		
		Item 2. Observe that every $\alpha$-co-H\"{o}lder map $f\from X \to Y$ with constant $c$ is injective and the inverse map $f^{-1} \from f(X) \to X$ is $(1/\alpha)$-H\"{o}lder map with constant $c^{-1/\alpha}$ and use item 1.
		
		Item 3. Combine item 1 and 2.
	\end{proof}
	
	\begin{prop}\label{holderproduct}
		Let $X, Y, X', Y'$ be metric spaces and $\alpha > 0$.
		\begin{enumerate}
			\item If $f\from X \to X'$ and $g\from Y \to Y'$ are $\alpha$-H\"{o}lder maps, then $f \times g\from X \times Y \to X' \times Y'$ is also an $\alpha$-H\"{o}lder map.
			\item If $f\from X \to X'$ and $g\from Y \to Y'$ are $\alpha$-co-H\"{o}lder maps, then $f \times g\from X \times Y \to X' \times Y'$ is also an $\alpha$-co-H\"{o}lder map.
		\end{enumerate}
	\end{prop}
	\begin{proof}
		We now adopt max metric as a metric of product space:
		\[
		d((x_1, y_1), (x_2, y_2)) = \max \{d(x_1, x_2), d(y_1, y_2)\}\ (x_1, x_2 \in X, y_1, y_2 \in Y).
		\]
		Note that the above metric and other two metrics $d(x_1, x_2) + d(y_1, y_2)$ and $\sqrt{d(x_1, x_2)^2+d(y_1, y_2)^2}$ are Lipschitz equivalent.
		
		By the assumption, there are $c_1, c_2 > 0$ such that
		\begin{align*}
			d(f(x_1), f(x_2)) &\le c_1 d(x_1, x_2)^\alpha \\
			d(g(y_1), g(y_2)) &\le c_2 d(y_1, y_2)^\alpha
		\end{align*}
		Then, we have
		\begin{align*}
			\max \{ d(f(x_1), f(x_2)), d(g(y_1), g(y_2)) \} \le \max \{ c_1, c_2\} \max\{d(x_1, x_2), d(y_1, y_2)\}^\alpha.
		\end{align*}
		So item (1) is proved. Item (2) can be shown by same argument.
	\end{proof}
	
	\begin{prop}\label{coholderfromcantortointerval}
		For every $\alpha \in (1, \infty)$, there is a $\alpha$-co-H\"{o}lder map $f\from 2^\omega \to [0, 1]$.
	\end{prop}
	\begin{proof}
		Put $\beta = 2^{-\alpha}$. Then $0 < \beta < 1/2$.
		Define $f \from 2^\omega \to [0, 1]$ by
		\[
		f(x) = (1-\beta) \sum_{n \in \omega} \beta^n x(n).
		\]
		Let $x \ne y \in 2^\omega$ and $n_0 = \min \{ n \in \omega : x(n) \ne y(n) \}$.
		Then
		\begin{align*}
			d(f(x), f(y)) &= (1-\beta) \abs*{\sum_{n \in \omega} \beta^n (x(n) - y(n))} \\
			&\ge (1-\beta) \left(\abs{\beta^{n_0}} - \abs*{\sum_{n > n_0} \beta^n (x(n) - y(n))}\right) \\
			&\ge (1-\beta) (\beta^{n_0} - \beta^{n_0+1}/(1-\beta)) \\
			&= (1-2\beta) \beta^{n_0} \\
			&= (1-2\beta) (2^{-n_0})^{-\log_2 \beta} \\
			&= (1-2\beta) (d(x, y))^{\alpha} \\
		\end{align*}
	\end{proof}
	
	\begin{prop}\label{coholderfromintervaltocantor}
		There is a $1$-co-H\"{o}lder map $f\from [0, 1] \to 2^\omega$.
	\end{prop}
	\begin{proof}
		Define $g\from 2^\omega \to [0, 1]$ by
		\[
		g(x) = \sum_{n \in \omega} \frac{x(n)}{2^{n+1}}.
		\]
		Let $f \from [0, 1] \to 2^\omega$ satisfies $g \circ f = \id$.
		In order to show $f$ is a $1$-co-H\"{o}lder map, it suffices to prove that
		\[
		(\forall x, y \in 2^\omega) (d(g(x), g(y)) \le d(x, y)).
		\]
		Fix $x, y \in 2^\omega$. If $x = y$, then it is obvious.
		So assume that $x \ne y$ and let $n_0$ be the minimum $n$ that $x(n) \ne y(n)$.
		Then
		\begin{align*}
			d(g(x), g(y)) &= \abs*{\sum_{n \in \omega} \frac{x(n) - y(n)}{2^{n+1}}} \\
			&\le \frac{1}{2^{n_0+1}} + \sum_{n > n_0} \frac{1}{2^{n+1}} \\
			&= 1/2^{n_0} \\
			&= d(x, y).
		\end{align*}
	\end{proof}

	\begin{fact}\label{dimofeuclid}
		For every $d \in \omega \setminus \{0\}$, we have $\dimh([0, 1]^d) = d$.
	\end{fact}
	
	\begin{prop}\label{dimofcantoris1}
		$\dimh(2^\omega) = 1$.
	\end{prop}
	\begin{proof}
		By Proposition \ref{coholderfromintervaltocantor},
		\[
		\dimh(2^\omega) \ge \dimh [0, 1] = 1.
		\]
		On the other hand, for every $\alpha > 1$, by Proposition \ref{coholderfromcantortointerval},
		\[
		\dimh(2^\omega) \le \alpha \dimh [0, 1] = \alpha.
		\]
		So $\dimh(2^\omega) \le 1$.
		
	\end{proof}
	
	\begin{prop}\label{biholdercantor}
		For every $d \in \omega \setminus \{0\}$, there is a $(1/d)$-bi-H\"{o}lder map $f\from 2^\omega \to (2^\omega)^d$.
	\end{prop}
	\begin{proof}
		Define f by
		\[
		f(x)(i)(m) = x(m \cdot d + i) \text{ (for $x \in 2^\omega, i < d$ and $m \in \omega$)}.
		\]
		Let $x \ne y \in 2^\omega$ and $n_0 = \min \{ n \in \omega : x(n) \ne y(n) \}$.
		And take $i_0 < n$ and $m_0 \in \omega$ such that $n_0 = m_0 \cdot d + i_0$.
		Then $f(x)(i_0)(m_0) \ne f(y)(i_0)(m_0)$ and $f(x)(i)(m) = f(y)(i)(m)$ for any $i < n$ and $m < m_0$.
		So
		\[
		d(f(x), f(y)) = 2^{-m_0}.
		\]
		Now we have
		\[
		d m_0 \le n_0 \le d (m_0 + 1)
		\]
		So
		\[
		n_0 / d - 1 \le m_0 \le n_0 / d.
		\]
		Thus
		\[
		d(x, y)^{1/d} \le d(f(x), f(y)) \le 2 d(x, y)^{1/d}.
		\]
	\end{proof}
	
	\begin{prop}\label{coholderandhdz}
		Let $f \from X \to Y$ be an $\alpha$-co-H\"{o}lder map for some $\alpha > 0$.
		Then
		\[
		\non(\HDZ_X) \ge \non(\HDZ_Y) \text{ and } \cov(\HDZ_X) \le \cov(\HDZ_Y).
		\]
	\end{prop}
	\begin{proof}
		Define a Galois-Tukey morphism $(f, f^{-1}): \Cov(\HDZ_X) \to \Cov(\HDZ_Y)$.
		Since $f^{-1}$ is $1/\alpha$-H\"{o}lder, this satisifies
		\[
		A \in \HDZ_Y \Rightarrow f^{-1}(A) \in \HDZ_X. \qedhere
		\]
	\end{proof}

	\begin{lem}\label{intervalandr}
		For every $d \in \omega \setminus \{0\}$, $\non(\HDZ_{[0, 1]^d}) = \non(\HDZ_{\R^d})$ and $\cov(\HDZ_{[0, 1]^d}) = \cov(\HDZ_{\R^d})$.
	\end{lem}
	\begin{proof}
		Since $[0, 1]^d \subset \R^d$, it is clear that $\non(\HDZ_{\R^d}) \le \non(\HDZ_{[0, 1]^d})$ and $\cov(\HDZ_{[0, 1]^d}) \le \cov(\HDZ_{\R^d})$.
		
		Now we show $\non(\HDZ_{[0, 1]^d}) \le \non(\HDZ_{\R^d})$.
		Let $A \not \in \HDZ_{\R^d}$.
		Then by $\sigma$-additivity of Hausdorff measures, we have $A \cap [-n, n]^d \not \in \HDZ_{[-n, n]^d}$ for some $n \in \omega$.
		Then $\abs{A \cap [-n, n]^d} \ge \non(\HDZ_{[-n, n]^d}) = \non(\HDZ_{[-1, 1]^d})$.
		So $\abs{A} \ge \non(\HDZ_{[-1, 1]^d})$.
		Thus $\non(\HDZ_{\R^d}) \ge \non(\HDZ_{[0, 1]^d})$.
		
		Next we show $\cov(\HDZ_{\R^d}) \le \cov(\HDZ_{[0, 1]^d})$.
		Take $\mathcal{F} \subset \HDZ_{[0, 1]^d}$ of size $\cov(\HDZ_{[0, 1]^d})$ such that $\bigcup \mathcal{F} = [0, 1]^d$.
		Define
		\[
		\mathcal{G} = \{ \bigcup_{n \in \omega} \operatorname{scale}_n``(X) : X \in \mathcal{F} \},
		\]
		where $\operatorname{scale}_n \from [0, 1]^d \to [-n, n]^d; \boldsymbol{x} \mapsto 2n \boldsymbol{x} - (n, n, \dots, n)$.
		Then $\mathcal{G}$ is a subset of $\HDZ_{\R^d}$ and of size $\le \cov(\HDZ_{[0, 1]^d})$ and satisfies $\bigcup \mathcal{G} = \R^d$.
	\end{proof}

	\begin{thm}\label{replacingspaces}
		For all $d \in \omega \setminus \{0\}$, $\non(\HDZ) = \non(\HDZ_{\R^d})$ and $\cov(\HDZ) = \cov(\HDZ_{\R^d})$.
	\end{thm}
	\begin{proof}
		By Lemma \ref{intervalandr}, it suffices to show that $\non(\HDZ) = \non(\HDZ_{[0,1]^d})$ and $\cov(\HDZ) = \cov(\HDZ_{[0, 1]^d})$.
		
		By Proposition \ref{coholderfromintervaltocantor}, there is a $1$-co-H\"{o}lder $[0, 1] \to 2^\omega$.
		Then by Proposition \ref{holderproduct}, there is a $1$-co-H\"{o}lder $[0, 1]^d \to (2^\omega)^d$.
		By Proposition \ref{biholdercantor}, there is a $d$-bi-H\"{o}lder $(2^\omega)^d \to 2^\omega$.
		Composing these maps, we obtain $d$-co-H\"{o}lder map $[0,1]^d \to 2^\omega$.
		So by Proposition \ref{coholderandhdz}, we have $\non(\HDZ_{[0,1]^d}) \ge \non(\HDZ)$ and $\cov(\HDZ_{[0,1]^d}) \le \cov(\HDZ)$.
		
		On the other hand by Proposition \ref{coholderfromcantortointerval}, there is a $2$-co-H\"{o}lder map $2^\omega \to [0,1]$.
		Then by Proposition \ref{holderproduct}, there is a $2$-co-H\"{o}lder $(2^\omega)^d \to ([0,1])^d$.
		By Proposition \ref{biholdercantor}, there is a $(1/d)$-bi-H\"{o}lder $2^\omega \to (2^\omega)^d$.
		Composing these maps, we obtain $(2/d)$-co-H\"{o}lder map $2^\omega \to [0,1]^d$.
		So by Proposition \ref{coholderandhdz}, we have $\non(\HDZ) \ge \non(\HDZ_{[0,1]^d})$ and $\cov(\HDZ) \le \cov(\HDZ_{[0,1]^d})$.
	\end{proof}
	
	\begin{conj}\label{conjhdz}
		\begin{enumerate}
			\item For every compact Polish space $X$ with $0 < \Haus^s(X) < \infty$ for some $s > 0$, $\non(\HDZ_X) = \non(\HDZ)$ and $\cov(\HDZ_X) = \cov(\HDZ)$.
			%\item For every metric space $X$ which is isomorphic to the Cantor space with $0 < \Haus^s(X) < \infty$ for some $s > 0$, $\non(\HDZ_X) = \non(\HDZ)$ and $\cov(\HDZ_X) = \cov(\HDZ)$.
		\end{enumerate}
	\end{conj}
	
	\section{Hausdorff measure zero ideals and Yorioka ideals}
	
	\begin{lem}\label{hausinfty}
		For a gauge function $f$ and $A \subset X$, if $\Haus^f_\infty(A) = 0$, then $\Haus^f(A) = 0$.
	\end{lem}
	\begin{proof}
		By $\Haus^f_\infty(A) = 0$, we have
		\begin{align}
			(\forall \varepsilon > 0)(\exists \seq{C_n : n \in \omega})(A \subset \bigcup_n C_n \AND \sum_n f(\diam(C_n)) \le \varepsilon). \label{hausfinfty}
		\end{align}
		Let $\varepsilon, \delta > 0$.
		Take $\delta' < \delta$ such that $f(\delta') < f(\delta)$.
		Put $\varepsilon' = \min \{ \varepsilon, f(\delta') \}$.
		Then by (\ref{hausfinfty}), we can take $\seq{C_n : n \in \omega}$ such that
		\[
		\sum_n f(\diam(C_n)) \le \varepsilon'.
		\]
		Then for each $n$, we have
		\[
		f(\diam(C_n)) \le \varepsilon' \le f(\delta') < f(\delta).
		\]
		Since $f$ nondecreasing we have
		\[
		\diam(C_n) \le \delta.
		\]
		So we have showed
		\[
		(\forall \varepsilon, \delta > 0)(\exists \seq{C_n : n \in \omega})(A \subset \bigcup_n C_n \AND \sum_n f(\diam(C_n)) \le \varepsilon \AND (\forall n)(\diam(C_n) \le \delta)).
		\]
		That is, we showed $\Haus^f(A) = 0$.		
	\end{proof}
	
	\begin{defi}
		For a monotone function $e \in \omega^\omega$ that goes to $\infty$, we define a gauge function $e^*$ by
		\[
		e^*(2^{-k}) = 2^{-e(k)} \text{ for all  $k \in \omega$}.
		\]
		Define the value of $e^*(s)$ for $s$ being not a form of $2^{-k}$ by linear interpolation.	
	\end{defi}
	
	\begin{lem}\label{neandjh}
		Suppose that $e, h \in \omega^\omega$ nondecreasing satisfy $e(l) \le \min \{ n : l < h(2^n)\}$ for all $l \in \omega$.
		Then $\scrN^{e^*} \subset \scrJ_h$.
	\end{lem}
	\begin{proof}
		Let $A \in \scrN^{e^*}$.
		Then, for each $n \in \omega$, we can take $\sigma_n \in (\twototheltomega)^\omega$ such that $A \subset \bigcup_i [\sigma_n(i)]$ and $\sum_i 2^{-e(|\sigma_n(i)|)} < 2^{-n-2}$.
		Let $\sigma \in (\twototheltomega)^\omega$ be the enumeration of $\{ \sigma_n(i) : n, i \in \omega \}$ in ascending order of length.
		
		It is clear that $A \subset [\sigma]_\infty$.
		So we shall prove that $|\sigma(k)| \ge h(k)$ for all $k$.
		Assume that $|\sigma(k)| < h(k)$ for some $k$.
		For every $m \le k$,
		\[
		|\sigma(m)| \le |\sigma(k)| < h(k) \le h(2^{n_0}),
		\]
		where $n_0 = \ceil{\log_2 k}$.
		So for $m < k$ we obtain
		\[
		e(\abs{\sigma(m)}) \le \min \{ n : \abs{\sigma(m)} < h(2^n) \} \le n_0.
		\]
		Then we have
		\[
		\sum_{m < k} 2^{-e(|\sigma(m)|)} \ge \sum_{m < k} 2^{-n_0} = k \cdot 2^{-n_0} \ge 2^{n_0-1} 2^{-n_0} = 1/2.
		\]
		On the other hand, by $\sum_i 2^{-e(|\sigma_n(i)|)} < 2^{-n-2}$ for all $n$, we have $\sum_{k \in \omega} 2^{-e(|\sigma(k)|)} < 1/2$. It is a contradiction.
	\end{proof}
	
	\begin{lem}\label{yoriokaandhausideals}
		Let $e, c, h \in \omega^\omega$.
		Let $\seq{I_n : n \in \omega}$ be the interval partition such that $\abs{I_n} = h(n)$.
		Let $g_{c, h} \from \omega \to \omega$ be defined by $g_{c,h}(k) = \floor{\log_2 c(n)}$ whenever $k \in I_n$.
		Suppose that $e(g_{c, h}(n)) \ge 2 \log_2 n$ for all $n \in \omega$.
		Then $\frakv^\exists_{c, h} \le \non(\scrN^{e^*})$ and $\cov(\scrN^{e^*}) \le \frakc^\exists_{c, h}$.
	\end{lem}
	\begin{proof}
		This proof is based on \cite[Lemma 2.4]{klausner2019different}.
		We construct a Galois-Tukey morphism $(\varphi_-, \varphi_+) \from \Cov(\scrN^{e*}) \to \wLc(c, h)$.
		For each $n \in \omega$, let $\iota_n \from 2^{\floor{\log_2 c(n)}} \to c(n)$ be an injective map.
		Define $\varphi_-$ by $\varphi_-(y)(n) = \iota_n(y \upharpoonright \floor{\log_2 c(n)})$.
		For $S \in S(c, h)$, enumerate the members of $S$ by $S = \{ m^S_{n, k} : k \in I_n \}$.
		For $k \in I_n$, put
		\[
		\sigma_S(k) = \begin{cases}
			\iota_n^{-1}(m^S_{n, k}) & \text{(if $m^S_{n, k} \in \range \iota_n$)} \\
			(0)^{\floor{\log c(n)}} & \text{(otherwise).}
		\end{cases}
		\]
		Here $(0)^{\floor{\log c(n)}}$ denotes the zero sequence of length $\floor{\log c(n)}$.
		Define $\varphi_+(S) = [\sigma_S]_\infty$.
		
		Then clearly $\varphi_-(y) \in^\infty S \rightarrow y \in \varphi_+(S)$.
		Moreover, we have
		\begin{align*}
			\Haus_\infty^{e^*}([\sigma_S]_\infty) &= \Haus_\infty^{e^*}(\bigcap_n \bigcup_{m \ge n} [\sigma_S(m)]) \\
			&\le \Haus_\infty^{e^*}(\bigcup_{m \ge n} [\sigma_S(m)]) \\
			&\le \sum_{m \ge n} \Haus_\infty^{e^*}([\sigma_S(m)]) \\
			&\le \sum_{m \ge n} 2^{-e(g_{c,h}(m))} \\
			&\le \sum_{m \ge n} 1/m^2 \to 0\ (n \to \infty).
		\end{align*}
		So $[\sigma_S]_\infty \in \scrN^{e^*}$ by Lemma \ref{hausinfty}.
	\end{proof}

	\begin{lem}
		Let $e, g \in \omega^\omega$.
		Suppose that $e(g(i)) \ge 2i$ for all but finitely many $n$.
		Then $\scrJ_g \subset \scrN^{e^*}$.
	\end{lem}
	\begin{proof}
		Let $A \in \scrJ_g$.
		Then we can take $\sigma \in (\twototheltomega)^\omega$ such that $\height \sigma = g$ and $A \subset [\sigma]_\infty$.
		Let $\varepsilon > 0$.
		Now we have
		\[
		(\forall^\infty i) (e(\abs{\sigma(i)}) = e(g(i)) \ge 2i).
		\]
		So
		\[
		(\forall^\infty i) (\ 2^{-e(\abs{\sigma(i)})} \le 2^{-2i} \le 2^{-i} \varepsilon).
		\]
		
		Modifying the first finitely many terms in $\sigma$, we have
		\[ (\forall i) (2^{-e(\abs{\sigma(i)})} \le 2^{-i} \varepsilon).\]
		So
		\[ \sum_i 2^{-e(\abs{\sigma(i)})} \le \varepsilon.\]
		Thus, $\Haus^{e^*}_\infty(A) \le \epsilon$. Since $\epsilon > 0$ is arbitrary, we have $A \in \scrN^{e^*}$ by Lemma \ref{hausinfty}.
	\end{proof}
	
	\begin{cor}
		\begin{enumerate}
			\item For every gauge function $f$, there is an increasing function $g\in \omega^\omega$ such that $\scrI_g \subset \scrN^f$.
			\item For every increasing function $g \in \omega^\omega$, there is a gauge function $f$ such that $\scrN^f \subset \scrI_g$. \qed
		\end{enumerate}
	\end{cor}
	
	\begin{thm}\label{iidhdznull}
		$\scrI_\id \subsetneq \HDZ$.
	\end{thm}
	\begin{proof}
		To show $\scrI_\id \subset \HDZ$, let $A \in \scrI_\id$.
		Then we can take $f \gg \id$ and $\sigma \in (\twototheltomega)^\omega$ such that $A \subset [\sigma]_\infty$ and $\height \sigma = f$.
		Let $s, \varepsilon > 0$.
		By $f \gg \id$,
		\[
		(\forall^\infty i) \left(f(i) \ge i^2 \ge \frac{i+\log_2 (1/\epsilon)}{s}\right).
		\]
		Take $i_0 \in \omega$ such that $
		(\forall i \ge i_0) (f(i) \ge \frac{i+\log_2 (1/\epsilon)}{s})$. Now define $\tau \in (\twototheltomega)^\omega$ as
		\[
		\tau(i) = \begin{cases} (0)^{\ceil{(i+\log_2 (1/\varepsilon))/s}} & (i < i_0), \\  \sigma(i) & (i \ge i_0). \end{cases}
		\]
		If $x \in 2^\omega$ satisfies $(\exists^\infty i) (\sigma(i) \subset x)$ then $(\exists i) (\tau(i) \subset x)$. Thus $A \subset [\sigma]_\infty \subset \bigcup_i [\tau(i)]$.
		
		Also, by $|\tau(i)| \ge (i+\log_2 (1/\varepsilon))/s$, $2^{-|\tau(i)| s} \le \varepsilon / 2^i$.
		Thus we have $\sum_i 2^{-|\tau(i)| s} \le \varepsilon$.
		Therefore $A \in \HDZ$.
		
		For $\HDZ \setminus \scrI_{\id} \ne \emptyset$, take $A = \{ x \in 2^\omega : (\forall n \in \omega) \text{($x \upharpoonright I_n$ is constant)}\}$, where $I_n = [n^2, (n+1)^2)$.
		To show $A \in \HDZ \setminus \scrI_\id$, first define a tree $T$ as follows:
		\begin{align*}
			T_0 &= \{()\}, \\
			T_{n+1} &= \{ t \append (b)^{|I_n|} : t \in T_n, b \in 2 \}, \\
			T &= \bigcup_{n} T_n \downarrow.
		\end{align*}
		Here $T_n \downarrow$ denotes the downward closure of $T_n$.
		Clearly, the paths through $T$ is $A$.

		Note that 
		\[A = \bigcap_n \bigcup_{\sigma \in T_n} [\sigma].
		\]
		Let $s > 0$. Then we have
		\begin{align*}
			\Haus^s_\infty(A) &\le \Haus^s_\infty\left(\bigcup_{\sigma \in T_n} [\sigma]\right) \\
			&\le 2^{n} \cdot 2^{-n^2s} \\
			&\to 0 \text{ (as $n \to \infty$)}.
		\end{align*}
		So we get $\Haus^s(A) = 0$ by Lemma \ref{hausinfty}. Since $s > 0$ is arbitrary, we have $\dimh(A) = 0$. So $A \in \HDZ$.
		
		To show $A \not \in \scrI_\id$, assume that $A \in \scrI_\id$.
		Then we can take $\sigma \in (\twototheltomega)^\omega$ such that $\height \sigma \gg \id$ and $A \subset [\sigma]_\infty$. We may assume that $\range \sigma \subset T$.
		Take the natural bijection $\varphi\from \twototheltomega \to T$.
		Considering $\tau(n) = \varphi^{-1}(\sigma(n))$, we get $2^\omega \subset [\tau]_\infty$.
		Moreover, since $\varphi$ maps a node whose length is $n$ into a node whose length is $n^2$,
		\[
		\height \tau = \sqrt{\height \sigma} \gg \id.
		\]
		This implies $2^\omega \in \scrI_\id$, contradiction.
		Thus we get $A \not \in \scrI_\id$.
	\end{proof}
	
	\section{Additivity number and cofinality of $\HDZ$}
	
	\begin{fact}[{{\cite[Theorem 534B]{fremlin2008measure}}}]
		For every $0 < s < 1$, there is a Galois-Tukey isomorphism $\Cof(\scrN^s) \simeq \Cof(\nul)$. In particular $\cof(\scrN^s) = \cof(\nul)$ and $\add(\scrN^s) = \add(\nul)$.
	\end{fact}
	
	\begin{lem}
		Let $\seq{I_n : n \in \omega}$ be a sequence of ideals over a set $X$.
		Let $I_{n+1} \subset I_n$ for every $n$ and $I = \bigcap_n I_n$.
		Suppose that $\add(I_n) = \kappa$ for every $n$.
		Then $\add(I) \ge \kappa$.
	\end{lem}
	\begin{proof}
		Let $\lambda < \kappa$.
		Take $\seq{X_\alpha : \alpha < \lambda}$ with each $X_\alpha \in I$.
		Then by the assumption $\bigcup_{\alpha < \lambda} X_\alpha \in I_n$ for all $n$.
		Thus we have $\bigcup_{\alpha < \lambda} X_\alpha \in I$.
	\end{proof}
	
	\begin{cor}
		$\add(\nul) \le \add(\HDZ)$.
	\end{cor}
	
	\begin{thm}
		Let $\seq{I_n : n \in \omega}$ be a decreasing sequence of $\sigma$-ideals over a set $X$ and $I = \bigcap_m I_m$.
		Suppose that for each $m$, there is a Galois--Tukey morphism $(\varphi_m, \psi_m) \colon \Cof(I_m) \to \Lc(\omega, 2^\id)$.
		Then there is a Galois--Tukey morphism $(\varphi, \psi) \colon \Cof(I) \to \Lc(\omega, 2^\id)$.
	\end{thm}
	\begin{proof}
		Fix a bijection $\omega^{<\omega} \to \omega$ and let $\seq{a_0, \dots, a_n}$ denote the image of the $n$-tuple under this bijection.
		For $n \in \omega$, let $\pr_n \colon \omega \to \omega$ denote the $n$-th projection.
		Put 
		\[\PR_n \colon S(\omega, 2^\id) \to S(\omega, 2^\id); S \mapsto (\pr_n ``(S(n)) : n \in \omega).\]
		
		For $X \in I$, define $\varphi(X) \in \omega^\omega$ by
		\[
		\varphi(X)(m) = \seq{\varphi_0(X)(m), \dots, \varphi_m(X)(m)}.
		\]
		For $S \in S(\omega, 2^\id)$, define $\psi(S) \subset X$ by
		\[
		\psi(S) = \bigcup_{m \in \omega} \bigcap_{n \ge m} \psi_n(\operatorname{PR}_n(S)).
		\]
		Since $\varphi_n(\operatorname{PR}_n(S)) \in I_n$, we have $\psi(S) \in I$.
		Fix $X \in I$ and $S \in S(\omega, 2^\id)$ such that $\varphi(X) \in^* S$.
		Then
		\[
		(\forall^\infty i) (\seq{\varphi_0(X)(i), \dots, \varphi_i(X)(i)} \in S(i)).
		\]
		So
		\[
		(\forall^\infty i) (\forall n \le i) (\varphi_n(X)(i) \in \PR_n(S)(i)).
		\]
		Thus
		\[
		(\forall^\infty n) (\forall i \ge n) (\varphi_n(X)(i) \in \PR_n(S)(i)).
		\]
		So by the Galois-Tukeyness of $(\varphi_n, \psi_n)$,
		\[
		(\forall^\infty n) (X \subset \psi_n(\PR_n(S))).
		\]
		Thus, by the definition of $\psi$, we have
		\[
		X \subset \psi(S).
		\]
	\end{proof}
	
	\begin{cor}
		$\cof(\HDZ) \le \cof(\nul)$.
	\end{cor}
	
	\section{Separating uniformity of $\nul$ and $\scrN^s$}
	
	\begin{thm}\label{nonhdzltnonnull}
		\begin{enumerate}
			\item For every forcing poset $P$ with Laver property and $s \in (0, 1)$, $P \forces 2^\omega \cap V \not \in \scrN^s$.
			\item For every $s \in (0, 1)$, it is consistent with ZFC that $\non(\scrN^s) < \non(\nul)$.
			\item For every $0 < s < d$ with $d \in \omega$, it is consistent with ZFC that $\non(\scrN_{\R^d}^s) < \non(\nul)$.
		\end{enumerate}
	\end{thm}
	
	\begin{lem}\label{evalfrombelow}
		Let $0 < s < 1$ and $\sigma \in (\twototheltomega)^\omega$.
		Assume that $\sum_n 2^{-|\sigma(n)| s} \le 1$ and $\sigma$ is in ascending order of length.
		Then $|\sigma(n)| \ge (\log_2 n) / s - C_s$, where $C_s > 0$ is a constant depending only $s$.
	\end{lem}
	\begin{proof}
		By the assumption, there are at most $2^{ks}$ elements of length $k$ in $\sigma$.
		So for all $n\in \omega$, $\abs{\sigma(n)} \ge \alpha(n)$, where
		\[
		\alpha = \seq{0, \underbrace{1, \dots, 1}_{\ceil{2^s} \text{ terms}}, \underbrace{2, \dots, 2}_{\ceil{2^{2s}} \text{ terms}}, \underbrace{3, \dots, 3}_{\ceil{2^{3s}} \text{ terms}}, \dots}.
		\]
		Thus
		\begin{align*}
			k &= \alpha(1 + \ceil{2^s} + \dots + \ceil{2^{(k-1)s}}).
		\end{align*}
		So
		\[
		n \ge 1 + \ceil{2^s} + \dots + \ceil{2^{(k-1)s}} \Rightarrow \alpha(n) \ge k.
		\]
		Now for some $k_0$ we have for all $k \ge k_0$
		\begin{align*}
			1 + \ceil{2^s} + \dots + \ceil{2^{(k-1)s}} &\le 1 + 2^s + \dots + 2^{(k-1)s} + k \\
			&\le 1 + 2^s + \dots + 2^{(k-1)s} + 2^{ks} \\
			&= (2^{(k+1)s}-1)/(2^s-1).
		\end{align*}
		So for $k \ge k_0$,
		\[
		n \ge (2^{(k+1)s}-1)/(2^s-1)  \Rightarrow \alpha(n) \ge k.
		\]
		That is
		\[
		\log_2((2^s-1) n + 1)/s-1 \ge k  \Rightarrow \alpha(n) \ge k.
		\]
		So
		\[
		\alpha(n) \ge \log_2((2^s-1) n + 1)/s-2,
		\]
		provided that $\log_2((2^s-1) n + 1)/s-1 \ge k_0$.
		
		Thus for all but finitely many $n$ we have
		\begin{align*}
			\alpha(n) &\ge \log_2((2^s-1) n + 1)/s-2 \\
			&\ge \log_2((2^s-1) n)/s-2 \\
			&= \log_2 n / s + \log_2 (2^s-1) / s - 2
		\end{align*}
		
		So putting $C_s = 2 - \log_2 (2^s-1) / s$ gives
		\[
		\alpha(n) \ge \log_2 n / s - C_s.
		\]
		By adjusting $C_s$, we can make the above inequality hold for all n.
	\end{proof}
	
	\begin{lem}\label{ncofhausiszero}
		Let $A \subset 2^\omega$ and $s > 0$.
		If $\Haus^s(A) = 0$, then there is $\sigma \in (\twototheltomega)^\omega$ such that
		\[
		\sum_n 2^{-\abs{\sigma(n)}\cdot s} \le 1 \AND A \subset [\sigma]_\infty.
		\]
	\end{lem}
	\begin{proof}
		By $\Haus^s_\infty(A) = 0$, for each $n \in \omega$, we can take $\seq{C^n_m : m \in \omega}$ such that $\sum_m \diam(C^n_m)^s \le 2^{-(n+1)}$ and $A \subset \bigcup_{m \in \omega} C^n_m$.
		Expand each $C^n_m$ to basic open set $[\sigma^n_m]$ so that its diameter does not change.
		Let $\sigma \in (\twototheltomega)^\omega$ be an enumeration of $\seq{\sigma^n_m : n, m \in \omega}$.
		Then we have $\sum_n 2^{-\abs{\sigma(n)}\cdot s} \le 1$ and $A \subset [\sigma]_\infty$.
	\end{proof}
	
	\begin{lem}\label{covhdzisborel}
		For $0 < s < 1$, there is a Borel relational system $\Cov'(\scrN^s)$ that is equivalent to $\Cov(\scrN^s)$.
	\end{lem}
	\begin{proof}
		Define $H$ and $\triangleleft$ as follows
		\begin{align*}
			H =& \{ \sigma \in (\twototheltomega)^{\omega\times\omega} : (\forall m \in \omega)(\sum_n 2^{-\sigma(n, m) \cdot s} \le 1/(m+1)) \}, \\
			x \triangleleft \sigma \iff& (\forall m)(x \in \bigcup_{n \in \omega} [\sigma(n, m)]).
		\end{align*}
		Then $\Cov'(\scrN^s) = (2^\omega, H, \triangleleft)$ suffices.
	\end{proof}
	
	\begin{proof}[Proof of Theorem \ref{nonhdzltnonnull}]
		To show item (1), fix $\dot{\sigma}$ and $p \in \P$ such that $p \forces \dot{\sigma} \in (\twototheltomega)^\omega$ and $p \forces \sum_n 2^{-s |\dot{\sigma}(n)|} \le 1/2$.
		Fix $q \le p$.
		Put $\beta(n) = \floor{(\log_2 n) / s - C_s}$ where $C_s$ is the constant from Lemma \ref{evalfrombelow}.
		Define $\dot{\tau}$ so that $p \forces \dot{\tau}(n) = \dot{\sigma}(n) \upharpoonright \beta(n)$.
		
		By the Laver property, we can take $r \le q$ and $S \in \prod_{n \in \omega} [{}^{\beta(n)} 2]^{\le n^{(1-s)/2}}$ such that $r \forces (\forall^\infty n) (\dot{\tau}(n) \in S(n))$.
		Let $X$ be the Borel code of $\bigcap_{k \in \omega} \bigcup_{n \ge k} \bigcup \{[t] : t \in S(n) \}$.
		Then we have $r \forces [\dot{\tau}]_\infty \subseteq \hat{X}$.
		
		Now we have the following:
		\begin{align*}
			\Haus^{(1+s)/2}_\infty(\hat{X}) &\le \Haus^{(1+s)/2}_\infty(\bigcup_{n \ge k} \bigcup \{ [t] : t \in S(n) \}) \\
			&\le \sum_{n \ge k} n^{(1-s)/2} (2^{-\beta(n)})^{(1+s)/2} \\
			&\le 2^{(C_s + 1) (1+s)/2} \sum_{n \ge k} n^{(1-s)/2} n^{-(1+s)/(2s)} \\
			&= 2^{(C_s + 1) (1+s)/2} \sum_{n \ge k} n^{-(1/2)(s+1/s)} \to 0 \text{ (as $k \to \infty$)}.
		\end{align*}
		We used $(1/2)(s + 1/s) > 1$ in the last equation.
		Thus $\Haus^{(1+s)/2}(\hat{X}) = 0$.
		Since $(1+s)/2 < 1$ and $\dimh(2^\omega) = 1$, we have $\hat{X} \ne 2^\omega$.
		
		Then we can take $x \in 2^\omega \setminus \hat{X}$ in $V$.
		Then by absoluteness we have also $r \forces x \in 2^\omega \setminus \hat{X}$.
		By $r \forces [\dot{\sigma}]_\infty \subset [\dot{\tau}]_\infty \subset \hat{X}$, we have $r \forces x \not \in [\dot{\sigma}]_\infty$.
		Therefore we have $\forces (\forall \sigma\in(\twototheltomega)^\omega) (\sum_n 2^{-|\sigma(n)|s} \le 1 \Rightarrow 2^\omega \cap V \not \subseteq [\dot{\sigma}]_\infty)$.
		So by Lemma \ref{ncofhausiszero}, we obtain $\forces \Haus^s(2^\omega \cap V) > 0.$
		
		For item (2), consider $\omega_2$-step countable support iteration of Mathias forcing over a model of CH.
		In this model, by the item (1) and Lemma \ref{covhdzisborel}, $\non(\scrN^s) = \aleph_1$ whereas $\non(\nul) = \aleph_2$.
		
		For item (3), use item (2) and Proposition \ref{coholderfromcantortointerval}.
		In detail, let $0 < s < d$ and put $s' = s(1+\varepsilon)/d < 1$ for some $\varepsilon > 0$.
		By Proposition \ref{holderproduct}, Proposition \ref{coholderfromcantortointerval} and Proposition \ref{biholdercantor}, there is a $(1+\varepsilon)/d$-co-H\"{o}lder map $f \from 2^\omega \to [0, 1]^d$.
		Take $A \subset 2^\omega$ such that $\abs{A} = \non(\scrN^s)$ and $\Haus^{s'}(A) > 0$.
		Then $\Haus^s(f(A)) \ge C \cdot \Haus^{s'}(A) > 0$ for some constant $C > 0$ by Proposition \ref{holderhaus}. Now we have $\abs{f(A)} \le \abs{A} = \non(\scrN^s)$, so $\non(N_{[0,1]^d}^s) \le \non(N^{s'})$.
		Thus in the same model of (2), we have $\non(\scrN_{\R^d}^s) < \non(\nul)$.
	\end{proof}
	
	\begin{rmk}
		The same consistency of Theorem \ref{nonhdzltnonnull} (2) was already proved in \cite{SHELAH2005403}.
		But the forcing posets are simpler in our work than in their work.
	\end{rmk}
	
	\section{Many different uniformity numbers of Hausdorff measure 0 ideals}
	
	In this section, we prove the following theorem.
	
	\begin{thm}\label{manydiffunifhaus}
		It is consistent with ZFC that there are $\aleph_1$ many cardinals of the form $\non(\scrN^f)$ below continuum.	
	\end{thm}
	We modify the proof that there are consistently many different uniformity numbers of Yorioka ideals from \cite{klausner2019different}.
	
	\begin{defi}
		\begin{enumerate}
			\item For $c, h \in \omega^\omega$, define $g_{c,h} \in \omega^\omega$ by
			\[
			g_{c,h}(k) = \floor{\log_2 c(n)} \text{ (whenever $k \in J_n$)}
			\]
			where $(J_n)_{n \in \omega}$ is the interval partition with $|J_n| = h(n)$ for all $n \in \omega$.
			\item For $b, g \in \omega^\omega$, define $f_{b, g} \in \omega^\omega$ by
			\[
			f_{b, g}(k) = \sum_{l \le n} \ceil{\log_2 b(l)} \text{ (whenever $k \in I_n$)}
			\]
			where $(I_n)_{n \in \omega}$ is the interval partition with $|I_n| = g(n)$ for all $n \in \omega$.
			\item For $f \in \omega^\omega$ increasing, define $e_f \in \omega^\omega$ by
			\[
			e_f(k) = \min\{n \in \omega : k < f(2^n)\}.
			\]
			
			\item For $c, h \in \omega^\omega$ define $c^{\triangledown h} \in \omega^\omega$ by
			\[
			c^{\triangledown h}(n) = \abs{[c(n)]^{\le h(n)}}.
			\]
		\end{enumerate}
	\end{defi}
	
	\begin{defi}[{{\cite[Definition 4.1]{klausner2019different}}}]
		Two functions $(n_k^-)_{k \in \omega}, (n_k^+)_{k \in \omega}$ of natural numbers $\ge 2$ are called \textbf{bounding sequences} if 
		\begin{enumerate}
			\item[(i)] $n_k^- \cdot n_k^+ < n_{k+1}^-$ for all $k \in \omega$, and
			\item[(ii)] $\lim_{k \to \infty} \log_{n_k^-} n_k^+ = \infty$.
		\end{enumerate}
		
		Given bounding sequences $(n_k^-)_{k \in \omega}, (n_k^+)_{k \in \omega}$, a family
		$\mathcal{F} = \{ (a_\alpha, d_\alpha, b_\alpha, g_\alpha, f_\alpha, c_\alpha, h_\alpha) : \alpha \in A  \}$
		of tuples of increasing functions in $\omega^\omega$
		is called \textbf{suitable} with respect to $(n_k^-)_{k \in \omega}, (n_k^+)_{k \in \omega}$ if it satisfies the following properties for all $\alpha \in A$:
		\begin{enumerate}
			\item[(S1)] For all $k \in \omega$, we have $a_\alpha(k), d_\alpha(k), b_\alpha(k), g_\alpha(k), b_\alpha^{\triangledown g_\alpha}(k), \frac{b_\alpha(k)}{g_\alpha(k)}, h_\alpha(k), c_\alpha^{\triangledown h_\alpha}(k) \in [n_k^-, n_k^+]$.
			\item[(S2)] $h_\alpha < c_\alpha$ and $\limsup_{k \to \infty} \frac{1}{d_\alpha(k)} \log_{d_\alpha(k)} (h_\alpha(k) + 1) = \infty$.
			\item[(S3)]$b_\alpha / g_\alpha \ge d_\alpha$.
			\item[(S4)] $a_\alpha \ge b_\alpha^{\triangledown g_\alpha}$.
			\item[(S5)] There is some $l > 0$ such that $f_{b_\alpha, g_\alpha} \le^* f_\alpha \circ \pow_l$.
			\item[(S6)] $f_\alpha \ll g_{c_\alpha, h_\alpha}$.
			\item[(S7)] For all $\beta \in A$ with $\beta \ne \alpha$,
			\[
			\lim_{k \to \infty} \min \left\{ \frac{c_\beta^{\triangledown h_\beta}(k)}{d_\alpha(k)}, \frac{a_\alpha(k)}{d_\beta(k)} \right\} = 0.
			\]
		\end{enumerate}
	\end{defi}
	
	\begin{fact}[{{\cite[Section 4 and 5]{klausner2019different}}}]\label{kmfact}
		\begin{enumerate}
			\item There are bounding sequences and there is suitable family $\mathcal{F}$ of continuum size with respect to them.
			\item Assume CH and let $(\kappa_\alpha : \alpha \in A)$ be a sequence of infinite cardinals such that $|A| \le \aleph_1$ and $\kappa_\alpha^\omega = \kappa_\alpha$ for all $\alpha \in A$.
			Given a family $\mathcal{F} =  \{ (a_\alpha, d_\alpha, b_\alpha, g_\alpha, f_\alpha, c_\alpha, h_\alpha) : \alpha \in A  \}$ satisfying (S1) and (S7) with respect to some bounding  sequences, there is a forcing poset that preserves all cardinals and forces
			\[
			\mathfrak{c}^\forall_{a_\alpha,d_\alpha} \le \kappa_\alpha \le \mathfrak{v}^\exists_{c_\alpha,h_\alpha}.
			\]
			for all $\alpha \in A$.
			If the family $\mathcal{F}$ is suitable, then
			\[
			\mathfrak{v}^\exists_{c_\alpha,h_\alpha} \le \non(\scrI_{f_\alpha}) \le \mathfrak{v}^\exists_{b_\alpha,g_\alpha} \le	\mathfrak{c}^\forall_{a_\alpha,d_\alpha}
			\]
			is a ZFC theorem. Thus the forcing poset forces
			\[
			\mathfrak{v}^\exists_{c_\alpha,h_\alpha} = \non(\scrI_{f_\alpha}) = \mathfrak{v}^\exists_{b_\alpha,g_\alpha} = \mathfrak{c}^\forall_{a_\alpha,d_\alpha} = \kappa_\alpha
			\]
			for all $\alpha \in A$.
		\end{enumerate}	
	\end{fact}
	
	\begin{defi}
		Given bounding sequences $(n_k^-)_{k \in \omega}, (n_k^+)_{k \in \omega}$, a family
		$\mathcal{F} = \{ (a_\alpha, d_\alpha, b_\alpha, g_\alpha, c_\alpha, h_\alpha, e_\alpha, u_\alpha) : \alpha \in A  \}$
		of tuples of increasing functions in $\omega^\omega$
		is called \textbf{modified suitable} with respect to $(n_k^-)_{k \in \omega}, (n_k^+)_{k \in \omega}$ if it satisfies (S1), (S2), (S3), (S4), (S7) and the following (MS1), (MS2) and (MS3) for all $\alpha \in A$:
		\begin{enumerate}
			\item[(MS1)] $e_\alpha = e_{u_\alpha}$.
			\item[(MS2)] $e_\alpha(g_{c_\alpha, h_\alpha}(k)) \ge 2 \log_2 k$ for all $k \in \omega$.
			\item[(MS3)] $f_{b_\alpha,g_\alpha} \le u_\alpha$.
		\end{enumerate}
	\end{defi}
	
	\begin{prop}\label{verysuitable1}
		For a modified suitable family $\mathcal{F} = \{ (a_\alpha, d_\alpha, b_\alpha, g_\alpha, c_\alpha, h_\alpha, e_\alpha, u_\alpha) : \alpha \in A  \}$, we have
		\[
		\mathfrak{v}^\exists_{c_\alpha,h_\alpha} \le \non(\scrN^{e_\alpha^*}) \le \non(\scrJ_{u_\alpha}) \le \mathfrak{v}^\exists_{b_\alpha,g_\alpha} \le \mathfrak{c}^\forall_{a_\alpha,e_\alpha}.
		\]
	\end{prop}
	\begin{proof}
		$\mathfrak{v}^\exists_{c_\alpha,h_\alpha} \le \non(\scrN^{e_\alpha^*})$ follows from (MS2) and Lemma \ref{yoriokaandhausideals}.
		$\non(\scrN^{e_\alpha^*}) \le \non(\scrJ_{u_\alpha})$ follows from (MS1) and Lemma \ref{neandjh}.
		$\non(\scrJ_{u_\alpha}) \le \mathfrak{v}^\exists_{b_\alpha,g_\alpha}$ follows from (MS3) and \cite[Lemma 2.5]{klausner2019different}.
		$\mathfrak{v}^\exists_{b_\alpha,g_\alpha} \le \mathfrak{c}^\forall_{a_\alpha,e_\alpha}$ follows from (S3), (S4) and \cite[Lemma 2.6]{klausner2019different}.
	\end{proof}
	
	\begin{prop}\label{verysuitable2}
		There are bounding sequences and there is a modified suitable family $\mathcal{F}$ of continuum size with respect to them.
	\end{prop}
	\begin{proof}
		First we build bounding sequences $(n_k^-)_{k \in \omega}, (n_k^+)_{k \in \omega}$ and a modified suitable family $\mathcal{F} = \{ (a, d, b, g, c, h, e, u) \}$ of size $1$ by recursion.
		Let $n_0^- = 2$ and $d(0) = 3$.
		Let $\seq{I_n : n \in \omega}$ and $\seq{J_n : n \in \omega}$ be interval partitions with $\abs{I_n} = g(n)$ and $\abs{J_n} = h(n)$.
		Define the component of $\mathcal{F}$ in the following order:
		\begin{enumerate}
			\item $h(k) = d(k)^{(k+1) d(k)}$,
			\item $g(k) = \max \{ (\max J_k)^2 - \min I_k + 1, h(k) + 1 \}$,
			\item $b(k) = 2^{g(k) + d(k)}$,
			\item $u(j) = \sum_{l \le k} \log_2 b(l) + j - \min I_k$ for $j \in I_k$,
			\item $c(k) = 2^{u((\max J_k)^2)-1}$,
			\item $a(k) = \max \{ c^{\triangledown h}(k), b^{\triangledown g}(k) \} + 1$,
			\item $n_k^+ = a(k)$,
			\item $n_{k+1}^- = n_k^- \cdot n_k^+ + 1$ and
			\item $d(k+1) = n_{k+1}^- + 1$.
		\end{enumerate}
		
		Item (1), (3) and (6) ensures (S2), (S3) and (S4) respectively.
		
		Item (4) ensures (MS3) since
		\begin{align*}
			u(j) &= \sum_{l \le k} \log_2 b(l) + j - \min I_k \\
			&\ge \sum_{l \le k} \log_2 b(l) \\
			&= f_{b,g}(j)
		\end{align*}
		for $j \in I_k$.
		Moreover this definition ensures $u$ is strictly increasing since
		\begin{align*}
			u(\max I_{k-1}) &= \sum_{l \le k-1} \log_2 b(l) + \max I_{k-1} - \min I_{k-1} \\
			&= \sum_{l \le k-1} \log_2 b(l) + g(k-1) - 1 \\
			&< \sum_{l \le k} \log_2 b(l) \\
			&= u(\min I_k).
		\end{align*}
		
		When we are done defining $(n_k^-)_{k \in \omega}, (n_k^+)_{k \in \omega}, a, d, b, g, c, h$ and $u$, we define $e$ by
		\[
		e = e_u.
		\]
		This ensures (MS1).
		
		Item (5) ensures (MS2) since
		\[
		g_{c,h}(k) = u((\max J_k)^2)-1 \ge u(k^2)-1
		\]
		for $k \in J_n$ and
		\begin{align*}
			e_u(g_{c,h}(k)) &= \min \{ n : g_{c,h}(k) < u(2^n) \} \\
			&\ge \min \{ n : u(k^2)-1 < u(2^n) \} \\
			&= \ceil{\log_2 k^2} \\
			&\ge 2 \log_2 k.
		\end{align*}
		
		Item (2) ensures that in item (5) we will not access $u$ with an invalid index.
		In fact, from item (2) we obtain
		\[
		g(k) \ge (\max J_k)^2 - \min I_k + 1.
		\]
		So we have
		\[
		(\max J_k)^2 \le \min I_k + g(k) - 1 = \max I_k.
		\]
		Thus when we are in (5), we already defined $u((\max J_k)^2)$.
		
		The above construction ensures
		\[ n_k^- < d(k) < h(k) < g(k)  < b(k) < b^{\triangledown g}(k) < a(k) = n_k^+, \]
		\[ n_k^- < d(k) < b(k)/g(k) < b(k) < n_k^+ \]
		and
		\[ n_k^- < d(k) < h(k) < c(k) < c^{\triangledown h}(k) < a(k) = n_k^+ \]
		So (S1) holds.
		
		Now we shall show how to construct a modified suitable family $\mathcal{F} = \{ (a_\alpha, d_\alpha, b_\alpha, g_\alpha, c_\alpha, h_\alpha, e_\alpha, u_\alpha) : \alpha \in 2^\omega  \}$ of size continuum.
		We construct approximations $\seq{(a_t, d_t, b_t, g_t, c_t, h_t, u_t) : t \in 2^{<\omega} }$ and then put for $\alpha \in 2^\omega$, $a_\alpha = \bigcup_{n \in \omega} a_{\alpha \upharpoonright n}$, etc.
		
		Let $\bar{0} = \seq{0, 0, 0, \dots}, \bar{1} = \seq{1, 1, 1, \dots}$.
		Let $\triangleleft$ denote the lexicographical order of $2^\omega$ and $2^n$ for $n\in \omega$.
		By recursion on  $n \in \omega$ we define $\seq{(a_t, d_t, b_t, g_t, c_t, h_t, u_t) : t \in 2^n }$.
		
		\begin{enumerate}
			\item Let $d_{\bar{0} \upharpoonright (n+1)}(n) > d_{\bar{0} \upharpoonright (n)}(n-1) \cdot a_{\bar{1} \upharpoonright (n)}(n-1) + 2$.
			\item When $d_t(n)$ is defined, put $d_{t^+}(n) = (n+1) a_t(n)$, where $t^+$ is the successor of $t$ in $\triangleleft$.
			\item Define $h_{t^+}(n), g_{t^+}(n), \dots, a_{t^+}(n)$ as in the construction of the modified suitable family of size $1$.
		\end{enumerate}
		
		Put $n_k^- = d_{\bar{0}}(k) - 1$ and $n_k^+ = a_{\bar{1}}(k)$.
		And put $e_\alpha = e_{u_\alpha}$ for $\alpha \in 2^\omega$.
		
		We finished the construction and have to check (S7).
		It suffices that we prove for $\alpha \mathrel{\triangleleft} \beta$, $\lim_{k \to \infty} \frac{a_\alpha(k)}{d_\beta(k)} = 0$.	
		Let $n$ be the minimum number such that $\alpha(n) < \beta(n)$.
		Then by the definition of $d_\beta$, we have $d_\beta(k) \ge (k+1) a_\alpha(k)$ for any $k \ge n$.
		Thus $\frac{a_\alpha(k)}{d_\beta(k)} \to 0$ (as $k \to 0$).
	\end{proof}
	
	By Fact \ref{kmfact} (2), Propositon \ref{verysuitable1} and \ref{verysuitable2}, we have Theorem \ref{manydiffunifhaus}.
	
	\section{Many different covering numbers of Hausdorff measure 0 ideals}
	
	The following fact was proved by Kamo and Osuga in \cite[Section 3]{osuga2014many}.
	
	\begin{fact}\label{kamoosugathm}
		Let $\delta$ be an ordinal and $\seq{\lambda_\alpha : \alpha < \delta}$ be a strictly increasing sequence of regular uncountable cardinals.
		Let $\kappa \ge \delta$ be a cardinal such that $\kappa = \kappa^{<\lambda_\alpha}$ for all $\alpha < \delta$.
		Let $\seq{b_\alpha, c_\alpha : \alpha < \delta}$ be a sequence of pairs of reals in $\omega^\omega$ such that $b_\alpha >^* c_\beta^H \cdot \id$ for all $\beta < \alpha < \delta$ and $b_\alpha >^* 2 \cdot \id$ for all $\alpha < \delta$, where $H = \seq{n^{n^2} : n \in \omega}$.
		Then there is a ccc forcing poset $\P$ such that
		\begin{align*}
			\P \forces ((\forall \alpha < \delta)(\frakc^\exists_{c_\alpha,H} \le \lambda_\alpha \le \frakc^\exists_{b_\alpha,1}) \AND \frakc = \kappa).
		\end{align*}
	\end{fact}
	
	\begin{thm}\label{manydiffcovhaus}
		It is consistent with ZFC that there are $\aleph_1$ many cardinals of the form $\cov(\scrN^f)$ below continuum.
	\end{thm}
	\begin{proof}
		Assume GCH.
		Put $\delta = \omega_1$. Put $\lambda_\alpha = \aleph_{\alpha+1}$ for $\alpha < \omega_1$. Put $\kappa = \aleph_{\omega_1 + 1}$.
		We define $\seq{b_\alpha, g_\alpha, e_\alpha, c_\alpha : \alpha < \omega_1}$ recursively so that
		\begin{enumerate}
			\item $b_\alpha >^* 2 \cdot \id$ for all $\alpha < \omega_1$,
			\item $b_\alpha >^* c_\beta^H \cdot \id$ for all $\beta < \alpha < \omega_1$,
			\item $g_\alpha(n) \ge \sum_{i \le n} \log_2 b_\alpha(i)$,
			\item $e_\alpha(n) = \min \{ m \in \omega : n < g_\alpha(2^m) \}$ and
			\item $c_\alpha$ satisfies $e_\alpha(g_{c_\alpha, H}(n)) \ge 2 \log_2(n)$ for all $n \in \omega$ and $\alpha < \omega_1$.
		\end{enumerate}
		Then, the assumption of Fact \ref{kamoosugathm} holds.
		So we can take a ccc forcing poset $\P$ such that
		\[
		\P \forces (\forall \alpha < \omega_1)(\frakc^\exists_{c_\alpha,H} \le \lambda_\alpha \le \frakc^\exists_{b_\alpha,1}).
		\]
		But by item (3) above and \cite[Lemma 1]{osuga2014many}, we have $\frakc^\exists_{b_\alpha,1} \le \cov(\scrJ_{g_\alpha})$.
		And item (4) and Lemma \ref{neandjh} gives $\cov(\scrJ_{g_\alpha}) \le \cov(\scrN^{e_\alpha^*})$.
		Item (5) and Lemma \ref{yoriokaandhausideals} gives $\cov(\scrN^{e_\alpha^*}) \le \frakc^\exists_{c_\alpha, H}$.
		
		Therefore we have
		\[
		\P \forces (\forall \alpha < \omega_1)(\frakc^\exists_{b_\alpha,1} = \cov(\scrJ_g) = \cov(\scrN^{e_\alpha^*}) = \frakc^\exists_{c_\alpha,H} = \lambda_\alpha).
		\]
		Especially we have
		\[
		\P \forces (\forall \alpha < \omega_1)(\cov(\scrN^{e_\alpha^*}) = \lambda_\alpha).
		\]
	\end{proof}
	
	\section{Discussion}
	
	\begin{rmk}
		Combining theorems in Zapletal's book \cite{zapletal2008forcing}, it turns out that, for every $s \in (0, 1)$, it is consistent with ZFC that $\cov(\nul) < \cov(\scrN^s)$.
		In fact, assume $\mathrm{CH}$. By \cite[Theorem 4.4.2]{zapletal2008forcing}, the idealized forcing $\P$ determined from $s$-dimensional Hausdorff measure $0$ ideal is proper. Consider $\omega_2$ step countable support iteration $\P_{\omega_2}$ of $\P$.
		By \cite[Theorem 4.4.8]{zapletal2008forcing}, $\P$ adds no splitting reals, thus $\P_{\omega_2}$ also adds no splitting reals by \cite[Corollary 6.3.8]{zapletal2008forcing}. Therefore we have $\P_{\omega_2} \forces (\cov(\nul) = \aleph_1 \AND \cov(\scrN^s) = \aleph_2)$.		
	\end{rmk}
	
	\begin{problem}
		\begin{enumerate}
			\item Is it consistent that $\non(\scrI_\id) < \non(\HDZ)$?
			\item Is it consistent that $\cov(\HDZ) < \cov(\scrI_\id)$?
		\end{enumerate}
	\end{problem}
	\begin{problem}
		\begin{enumerate}
			\item Is it consistent that $\add(\nul) < \add(\HDZ)$?
			\item Is it consistent that $\cof(\HDZ) < \cof(\nul)$?
		\end{enumerate}
	\end{problem}
	
	\section*{Acknowledgement}
	
	First, the author is grateful to his supervisor Yasuo Yoshinobu.
	Also, he would like to thank Yoshito Ishiki and Takayuki Kihara, who gave him advice in proving Theorem \ref{replacingspaces} and Lemma \ref{neandjh} respectively.
	Moreover, he is grateful to Jörg Brendle, who read and commented on this paper.
	
	\nocite{*}
	\printbibliography
	
\end{document}